\documentclass[12pt]{amsart}

\usepackage{amssymb}

\usepackage{enumitem}

\usepackage{graphicx}

\makeatletter
\@namedef{subjclassname@2020}{%
  \textup{2020} Mathematics Subject Classification}
\makeatother

\usepackage[T1]{fontenc}
\newtheorem{theorem}{Theorem}[section]
\newtheorem{corollary}[theorem]{Corollary}

\theoremstyle{definition}

\newtheorem{remark}[theorem]{Remark}

\numberwithin{equation}{section}

\begin{document}

%%%%% To ease editing, for IMPAN journals add:

\baselineskip=17pt

%%%%%%%%%%%%%%%%

\title[Ditkin sets for some functional spaces and applications to grand Lebesgue spaces]{Ditkin sets for some functional spaces and applications to grand Lebesgue spaces}

%\author[A. Kowalska]{Anna Kowalska}
\address{Istanbul Arel University Faculty of Science and Letters\\ Department of Mathematics and Computer Sciences\\
\'Büyükçekmece\\
34537 Istanbul, Turkey}
\email{turangurkanli@arel.edu.trl}

\author[A.T. Gurkanli]{Ahmet Turan Gürkanlı}
%\address{Institute of Mathematics\\ Jagiellonian University\\
%{\L}ojasiewicza 6\\
%30-348 Krak\'ow, Poland}
%\email{jk.nowak@im.uj.edu.pl}

\date{}

%\begin{abstract}
\begin{abstract}
 Let $G$ be a locally compact Abelian group with dual group $\widehat G $ and Haar measures  $d\mu$ and $\hat d\mu$ respectively. In this work we have proved that if $X$ is an essential Banach ideal in Beurling algebra $ L^1_{\omega}(G),$ then  a closed subset $E\subset \widehat G$ is a Ditkin set for $X$ if and only if $E$ is a Ditkin set for  $ L^1_{\omega}(G).$
Next, as an application we have investigated the Ditkin sets for grand Lebesgue space $L^{p),\theta}(G)$ and  Ditkin sets for $[L^p(G)]_{L^{p),\theta }}$, where  $[L^p(G)]_{L^{p),\theta }}$ is the closure of the set  $C_c(G)$ in  $L^{p),\theta}(G)$. 
 
%\end{abstract}
\
\end{abstract}



\maketitle

\section{Introduction}
 Let $G$ be a locally compact Abelian group with dual group $\widehat G $ and Haar measures  $d\mu$ and $\hat d\mu$  respectively. $C_{c}(G)$ denotes the space of all continuous complex valued functions on $G$ with compact support.  A real valued measurable and locally bounded function $\omega$ on a  locally compact Abelian group $G$ is said to be weight function if $\ w(x)\ge1, w(x+y)\leq w(x)w(y)$ for all $x,y\in G.$ For $p=1,$ $L^1_{\omega}(G)$ (or $L^1 (G,\omega))$   is a Banach algebra under convolution,  called Beurling algebra. It is known that  $  L_{\omega}^1(G)$ admits a bounded approximate identity. For $1\leq p<\infty,$ we define the weighted space $L_{\omega}^p(\Omega)$ (or $L^p (G,\omega)$)  with the norm 
\begin{align}
\left\Vert {f}\right\Vert_{p,\omega}=\left( \int_{\Omega }\left\vert f\right\vert ^{p }\omega(x)dx\right) ^{\frac{1}{p }}.
\end{align}
It a Banach space, (see \cite {rs}, \cite{w}).

Let $\Omega $ be a locally compact Hausdorff space and let $(\Omega,\mathfrak{B},\mu)$ be finite Borel measure space. The grand Lebesgue space $L^{p)}\left( \Omega \right) $ was introduced by Iwaniec-Sbordone in \cite {is}. This Banach space is defined by the norm
\begin{align}
\left\Vert {f}\right\Vert _{p)}=\sup_{0<\varepsilon \leq p-1}\left(
\varepsilon \int_{\Omega }\left\vert f\right\vert ^{p-\varepsilon }d\mu
\right) ^{\frac{1}{p-\varepsilon }}
\end{align}
where $1<p<\infty . $ For $0<\varepsilon \leq p-1,$ $L^{p}\left( \Omega \right) \subset L^{p)}\left( \Omega \right) \subset \L^{p-\varepsilon }\left( \Omega \right)$
 hold. For some properties and applications of $L^{p)}\left( \Omega \right) $  we refer to papers \cite {cr} ,\cite {fio},\cite {gis} ,\cite {gur1},\cite{gur2}.
 A generalization of the grand Lebesgue
spaces are the spaces $L^{p),\theta }\left( \Omega \right) ,$ $\theta \geq 0,$
defined by the norm
\begin{equation}
\left\Vert f\right\Vert _{p),\theta ,\Omega }=\left\Vert f\right\Vert
_{p),\theta }=\sup_{0<\varepsilon \leq p-1}\varepsilon ^{\frac{\theta }{
p-\varepsilon }}\left( \int_{\Omega }\left\vert f\right\vert ^{p-\varepsilon
}d\mu \right) ^{\frac{1}{p-\varepsilon }}=\sup_{0<\varepsilon \leq
p-1}\varepsilon ^{\frac{\theta }{p-\varepsilon }}\left\Vert {f}\right\Vert
_{p-\varepsilon };
\end{equation}%
when $\theta =0$ the space $L^{p),0}\left( \Omega \right) $ reduces to Lebesgue
space $L^{p}\left( \Omega \right) $ and when $\theta =1$ the space $%
L^{p),1}\left( \Omega \right) $ reduces to grand Lebesgue space $L^{p)}\left(
\Omega \right)$,(see \cite {cr} ,\cite {gis} ) . Again for $0<\varepsilon \leq p-1,$
\begin{equation*}
L^{p}\left( \Omega \right) \subset L^{p),\theta }\left( \Omega \right) \subset
L^{p-\varepsilon }\left( \Omega \right) 
\end{equation*}
hold.
 It is also known that the grand Lebesgue space  $L^{p),\theta }\left( \Omega \right) $ is not reflexive.
%\end{proof}
\section{Ditkin sets for some functional spaces}
Throughout this work we will assume that  $G$ is a locally compact Abelian group with dual group $\widehat G $ and Haar measures  $d\mu$ and $\hat d\mu$ respectively.

 Assume that $X$  is a Banach ideal in  $L^1_{\omega}(G).$ In the spirit of \cite {s}, \cite {rs} and \cite{gur3} we call a closed set $ E\subset \widehat G $ is a Ditkin set for  $X$ if for every  $f\in X$ such that the Fourier transform  $\hat f$ of the function $f$ vanishes on $E$ can be approximated in $X$ with functions $f\ast F$ such that $\hat F$ vanishes in some neighbourhood on $E.$ 

Authors also used names such as C-set and Wiener-Ditkin set in some sources. 

\begin{theorem}
Let  $L^1_{\omega}(G)$ be a  Beurling algebra  and let $X$  be a Banach ideal in $L^1_{\omega}(G).$ Then every Ditkin set for $X$ is a Ditkin set for  $L^1_{\omega}(G).$ 
\end{theorem}
\begin{proof}
 Since $X$ is a Banach ideal in $L^1_{\omega}(G), X$ is continuously embedded in $L^1_{\omega}(G).$ Then there exists $C>0$ such that $\|.\|_{1,w} \leq C\|.\|_X.$ Assume that a closed subset $E\subset \widehat G$ is a Ditkin set for $X.$ It is known that the Beurling's algebra  $L^1_{\omega}(G),$ admits a bounded approximate identity $(e_{\alpha})_{\alpha\in I},$ (see \cite {w},\cite{mn}). Let $f\in X$ such that $\hat f$ vanishes  on $E,$ and let $\varepsilon >0$ be given. Then there exists  $\alpha_0\in I$ such that 
\begin{align}
\|f-f\ast\ e_{\alpha_0}\|_{1,w}<\frac{\varepsilon}2.
\end{align}
Here $f\ast e_{\alpha_0}\in X$ and by the properties of Fourier transform, $(f\ast e_{\alpha_0}\widehat)=\hat f\hat e_{\alpha_0}=0$ on $E.$ Since $E$ is a Ditkin set for $X,$ there exists $F_1\in X$ such that $\hat F_1$ vanishes on a neighbourhood of $E$ and
\begin{align}
\|f\ast\ e_{\alpha_0}-F_1\ast(f\ast\ e_{\alpha_0})\|_X<\frac{\varepsilon}{2C}.
\end{align}
If we set $F=F_1\ast e_{\alpha_0},$ we have $F\in X$ and $\hat F=0$ on a neighbourhood of $E.$ Thus from (2.1) and (2.2) we obtain
\begin{align*}
\|f-f\ast F\|_{1,w}&\leq\|f-f\ast\ e_{\alpha_0}\|_{1,w}+\|f\ast\ e_{\alpha_0}-f\ast F\|_{1,w}\\&<\frac{\varepsilon}2+C\|f\ast\ e_{\alpha_0}-F_1\ast(f\ast\ e_{\alpha_0})\|_X=\frac{\varepsilon}2+C\frac{\varepsilon}{2C}=\varepsilon.
\end{align*}
This completes the proof.
\end{proof}
\begin{theorem}
Let  $L^1_{\omega}(G)$ be a  Beurling algebra  and let $X$  be an essential  Banach ideal in $L^1_{\omega}(G).$  Then every Ditkin set for $L^1_{\omega}(G)$ is a Ditkin set for $X.$ 
\end{theorem}
\begin{proof}
Assume that a closed subset $E\subset \widehat G$ is a Ditkin set for $L^1_{\omega}(G).$ Also  let $(e_{\alpha})_{\alpha\in I}$ be a bounded approximate identity of  Beurling's algebra  $L^1_{\omega}(G),$  and $\|e_{\alpha}\|_{1,w}<M,$ for all $\alpha\in I$ (see \cite {w},\cite{mn}). Since $X$  is an essential Banach ideal in $L^1_{\omega}(G),$ then for every $g\in X$
\begin{align}
lim_{I} g\ast e_{\alpha}=g
\end{align}
in X, (see  Corollary 15.3 in \cite{dw}).
 Let $f\in X$ such that $\hat f$ vanishes  on $E,$ and let $\varepsilon >0$ be given. From (2.3), there exists  $\alpha_1\in I$ such that
\begin{align}
\|f-f\ast\ e_ {\alpha_1}\|_X<\frac{\varepsilon}2.
\end{align}
Since $E$ is a Ditkin set for $L^1_{\omega}(G),$ there also exists  $F_1\in L^1_{\omega}(G)$ such that $\hat F_1$ vanishes on a neighbourhood of $E$ and
\begin{align}
\|f-f\ast {F_1}\|_{1,w}<\frac{\varepsilon}{2M}.
\end{align}
If we set $F=e_{\alpha_1}\ast F_1,$ then we have $F\in X$ and $\hat F$ vanishes on a neighbourhood of $E.$ Finally from (2.4) and (2.5), we have 
\begin{align*}
\|f-f\ast {F}\|_X&\leq\|f-f\ast\ e_{\alpha_1}\|_X+\|f\ast\ e_{\alpha_1}-f\ast F\|_X\\&=\|f-f\ast\ e_{\alpha_1}\|_X+\|f\ast\ e_{\alpha_1}-f\ast F_1\ast\ e_{\alpha_1}\|_X\\&\leq\|f-f\ast\ e_{\alpha_1}\|_X+\|f-f\ast {F_1}\|_X\|e_{\alpha_1}\|_{1,w}\\&<\frac{\varepsilon}2+M\frac{\varepsilon}{2M}=\varepsilon.
\end{align*}
This completes the proof.
\end{proof}

\begin{corollary}
  Assume that $X$ is an essential Banach ideal in $ L^1_{\omega}(G).$ Then  a closed subset $E\subset \widehat G$ is a Ditkin set for $X$ if and only if $E$ is a Ditkin set for  $ L^1_{\omega}(G).$
\end{corollary}
\begin{proof}
The proof is clear from Theorem 2.1 and Theorem 2.2.
\end{proof}
\begin{corollary}
 Assume that $X$  is an essential Banach ideal in $L^1(G).$ Then closed subgroups of  $\widehat G $ are Ditkin sets for $X.$\end{corollary}
\begin{proof}
It is known by Theorem 7.4.1 in \cite{rs} that closed subgroups of $\widehat G$ are Ditkin sets for $L^1 (G).$ Thus by Corollary 2.3, closed subgroups of $\widehat G$ are Ditkin sets for $X$.
\end{proof}

\section{Applications}
%\section{Applications}

\subsection{Ditkin sets for the generalized  grand Lebesgue space $L^{p),\theta }(G)$}
In this part we will investigate the Ditkin sets for  generalized grand Lebesgue space $L^{p),\theta }(G).$

\begin{theorem}
 Let $G$ be a locally compact Abelian group with dual group $\widehat G $ and Haar measures  $d\mu$ and $\hat d\mu$ respectively. If $\mu(G)$ is finite, then every Ditkin set for the generalized grand Lebesgue space $L^{p),\theta }(G)$ is a Ditkin set for $L^1 (G).$
\end{theorem}
\begin{proof}
Since $\mu(G)$ is finite, then $G$ is compact (see Theorem 4 in \cite {n}). It is known by Theorem 5 in  \cite{gur1}  that the generalized grand Lebesgue space  $L^{p),\theta }(G)$ is a Banach convolution module over $L^1(G).$ It is also proved in Theorem 4 in  \cite{gur1}  that the generalized grand Lebesgue space  $L^{p),\theta }(G)$ does not admit an approximate identity, bounded in $L^1 (G).$ Thus  $L^{p),\theta }(G)$ is not essential Banach module over $L^1(G)$ by Corollary 15.3 in \cite{dw}. Since we have the embedding   $L^{p),\theta }(G)\hookrightarrow L^1(G),$ then $L^{p),\theta }(G)$ is a Banach ideal in $ L^1(G)$ but not essential. Finally from Thorem 2.1 we observe that every Ditkin set for  $L^{p),\theta }(G)$ is a Ditkin set for $L^1(G)$.
\end{proof}
\begin{remark}
 Let $G$ be a locally compact Abelian group with dual group $\widehat G $ and Haar measures  $d\mu$ and $\hat d\mu$ respectively. Assume that $\mu(G)$ is finite.  It is known by Theorem 7.4.1 in \cite{rs} that closed subgroups of $\widehat G$ are Ditkin sets for $L^1 (G).$ Thus by Theorem 3.1, every Ditkin set for the generalized grand Lebesgue space $L^{p),\theta }(G)$ is a closed subgroups of $\widehat G.$ 
\end{remark}

\subsection{Ditkin sets for the space $[L^p(G)]_{L^{p),\theta }}$}

 It is known that the subspace ${C_{c}}(\Omega)$ of the space  $L^{p),\theta}\left( \Omega \right) $  is not dense in $L^{p),\theta}\left( \Omega \right)$.   Its closure $[L^p(\Omega)]_{L^{p),\theta }}$ consists of functions $f \in L^{p)}\left(\Omega \right)$  such that
\begin{equation}
\lim_{\varepsilon \rightarrow 0}\varepsilon ^{\frac{\theta }{p-\varepsilon }
}\left\Vert f\right\Vert _{p-\varepsilon}=0,  
\end{equation}
(see \cite{cr}).

 Now we will discuss the Ditkin sets for the space$[L^p(G)]_{L^{p),\theta }}$.
 \begin{theorem}
 Let $G$ be a locally compact Abelian group with dual group $\widehat G $ and Haar measures  $d\mu$ and $\hat d\mu$ respectively. Assume that $\mu(G)$ is finite. Then  a closed subset $E\subset \widehat G$ is a Ditkin set for $[L^p(G)]_{L^{p),\theta }}$ if and only if $E$ is a Ditkin set for  $ L^1(G).$
\end{theorem}
\begin{proof}
Since $\mu(G)$ is finite then $G$ is compact. It is known by Theorem 6 in  \cite{gur1} that $[L^p(G)]_{L^{p),\theta }}$ admits an approximate identity bounded in $L^1(G).$ Also it is known by Theorem 5 in  \cite{gur1}  that $[L^p(G)]_{L^{p),\theta }}$ is Banach convolution module over $L^1(G).$ Since  $[L^p(G)]_{L^{p),\theta }}\hookrightarrow L^1(G), $ from Theorem 5 in  \cite{gur1},  $[L^p(G)]_{L^{p),\theta }}$  is a Banach ideal in  $L^1(G).$ Then $[L^p(G)]_{L^{p),\theta }}$ is essential Banach convolution ideal in $L^1(G)$ by Corollary 15.3 in  \cite{dw}. Finally by Corollary 2.3,  a closed subset  $E\subset \widehat G$ is a Ditkin set for $[L^p(G)]_{L^{p),\theta }}$ if and only if $E$ is a Ditkin set for  $ L^1(G).$
\end{proof}
\begin{remark}
 Let $G$ be a locally compact Abelian group with dual group $\widehat G $ and Haar measures  $d\mu$ and $\hat d\mu$ respectively. Assume that $\mu(G)$ is finite. It is known by Theorem 7.4.1 in \cite{rs} that closed subgroups of $\widehat G$ are Ditkin sets for $L^1 (G).$ Thus by Corollary 1, closed subgroups of $\widehat G$ are Ditkin sets for $[L^p(G)]_{L^{p),\theta }}$.
\end{remark}


\begin{thebibliography}{[HD82]}

\bibitem{cr} R.E. Castillo and  H. Raferio, An Introductory Course in Lebesgue
Spaces, Springer International Publishing, Switzerland,  2016.
\bibitem{dw} R.S. Doran and J. Wichmann, Approximate Identities and Factorization in Banach Modules, Lecture Notes in Mathematics 768, Springer -Verlag, New York,1979.
%\bibitem{DK} Danelia N.,  Kokilashvili V. On the approximation of periodic
%functions within the frame of grand Lebesgue spaces. Bulletin of the
%Georgian national academy of sciences  2012; 6(2): 11--16.

%\bibitem{Fei1} Feichtinger HG, Banach convolution algebras of Wiener's type,
%Proc. Conf. " Functions, Series, Operators", Budapest,  1980, Colloquia
%Math. Soc. J. Bolyai, North Holland Publ. Co., Amsterdam- Oxford- New York 
%1983: 509--524.

%\bibitem{Fei2} Feichtinger HG and Gr\"{o}bner P. Banach Spaces of Distributions
%Defined by Decomposition Methods I. Math.Nachr.;  1985;  123:  97--120.

%\bibitem{Fio1} Fiorenza A, and Karadzhov GE. Grand and small Lebesgue spaces and
%their analogs, Journal for Analysis and its Applications 2004; 23 (4) :  657--681.

\bibitem{fio} A. Fiorenza, Duality and reflexity in grand Lebesgue spaces,
Collect. Math. 51 (2), ( 2000), 131--148.
%\bibitem{} Formica arXiv
\bibitem{gis} L. Greco, T. Iwaniec, C. Sbordone,  Inverting the p-harmonic
operator, Manuscripta Math. 92, (1997), 259--272.

\bibitem{gur1} A.T. Gurkanli, Inclusions and the approximate identities of the
 generalized grand Lebesgue spaces, Turkish J.  Math. 42, (2018), 3195--3203.

\bibitem{gur2} A.T. Gurkanli, On the grand Wiener amalgam spaces, Rocky Mountain J. Math., Vol.50, No.5, (2020), 1647-1659.
\bibitem{gur3} A.T. Gurkanli, Multipliers of some Banach Banach ideals and Wiener-Ditkin sets, Math. Slovaca,Vol. 55, No. 2, (2005), 237-248.


%\bibitem{h} Heil C, An Introduction to Weighted Wiener Amalgams, In: Wavelets
%and their Applications (Chennai, 2002), Allied Publishers, New Delhi, 2003:
%183-216.
%\bibitem{} Holland F. Square -summable positive-definite functions on real
%line, Linear operators Approx. II, Ser. Numer. Math. $\mathbf{25}$,
%Birkhauser, Basel, $1974$, $247-257.$

%\bibitem{} Holland F. Harmonic analysis on amalgams of $L^{p}$ and $\ell
%^{q},$ London Math. Soc. $\mathbf{10-2},$ $\left( 1975\right) ,$ $295-305.$

\bibitem{is} T. Iwaniec and C. Sbordone, On the integrability of the Jacobian under
minimal hypotheses, Arc. Rational Mech. Anal. (1992), 119, 129--143.
%\bibitem{sw} Stewart J, and Watson S. Which amalgams are convolution algebras,
%Proceedings of the American Mathematical Society, Volume 93, Number 
%4,1958,621-627.
%\bibitem{} Wiener N. Generalized Harmonic Analysis Tauberian Theorems, M.I.T. Press, 1966.
%\bibitem{kg} Kulak Ö, and Gurkanli AT. Bilinear Multipliers of small Lebesgue spaces,Turkish Journal of Mathematics, Vol.45:No.5, %2020.
\bibitem{mn} G.N.K. Murthy, and K.R. Unni,  Multipliers on weighted spaces, Functional analysis and its applications, International Conference, Madras, 1973, Lecture Notes in Mathematics, Springer Verlag, 399.
\bibitem{n} L. Nachbin, The Haar integral, Robert E.Krieger Publishing Co., INC, 1965.
\bibitem{s} J.D. Stegeman and Wiener -Ditkin sets for certain Beurling algebras, Monash. Math. 82, (1976), 295-307.
\bibitem{rs}H. Reiter and  J.D. Stegeman,  Classical Harmonic Analysis and Locally Compact Groups. Clarendon  Press, Oxford, 2000.
\bibitem{r} W. Rudin,  Fourier analysis on Groups. New York, Interscience, 1999.
\bibitem{w}H.C.  Wang, Homogeneous Banach Algebras, Lecture Notes in Pure and Applied Mathematics, Marcel Dekker INC., New York and Basel, 1977. 
%\bibitem{z} Zelazko W. On the algebras $L^{p}$ of locally compact groups,
%Colloquium Mathematicum, Vol. VIII, Fasc. 1,1961,115-120.
\end{thebibliography}
\end{document}